\documentclass[12pt]{article}  

\usepackage{amsmath}
\usepackage{amsfonts}
\usepackage{amssymb}
\usepackage{amsthm, cite}

\newcommand{\spann}{\operatorname{span}\nolimits}
\def\ad{\operatorname{ad}\nolimits}
\def\id{\operatorname{id}}

\def\so{\operatorname{so}}
\def\extr{\operatorname{extr}}
\def\rank{\operatorname{rank}}
\def\interior{\operatorname{int}}
\def\const{\operatorname{const}}

\def\f{\varphi}

\newcommand{\be}[1]{\begin{equation}\label{#1}}
\newcommand{\ee}{\end{equation}}

\newtheorem{theorem}{Theorem}
\newtheorem{lemma}{Lemma}
\newtheorem{corollary}{Corollary}

\newtheorem{example}{Example}

\title{\LARGE \bf
Periodic controls \\in step 2 strictly-convex \\sub-Finsler problems*
}

\author{Yuri L. Sachkov
\thanks{Sections 1--3 of this work are 
 supported by the Academy of Finland (grant 277923)
and by the European Research Council (ERC Starting Grant 713998 GeoMeG).
Sections 4--6 of this work are supported by the Russian Science Foundation 
under grant 17-11-01387 and performed in Ailamazyan Program Systems Institute 
of Russian Academy of Sciences
}
\thanks{Yuri Sachkov is with Program Systems Institute, Pereslavl-Zalessky, Russia and Department of Mathematics and Statistics, University of Jyv\"askyl\"a, Finland
        {\tt\small yusachkov@gmail.com}}%
}

\begin{document}

\maketitle

\begin{abstract}
We consider control-linear left-invariant time-optimal problems on step 2 Carnot groups  with strictly convex set of control parameters (in particular, sub-Finsler problems). 

We describe all linear-in-momenta Casimirs on the dual of the Lie algebra.

In the case of rank 3 Lie groups we describe the symplectic foliation on the dual of the Lie algebra. On this basis we show that extremal controls are either constant or periodic. 

Some related results for other Carnot groups are presented. 

\end{abstract}



\section{Introduction}\label{sec:intro}

We consider linear-in-controls time-optimal left-invariant problems on step 2 Carnot groups, with a strictly convex control set. In particular, this class of problems contains sub-Riemannian \cite{mont,  notes, ABB} and sub-Finsler \cite{ber, ber2, BBLDS, ALDS, ali-charlot} problems. Our aim is to characterize extremal controls.

It is enough to consider the free-nilpotent cases since any step-2 Carnot group is a quotient of a step-2
free-nilpotent Lie group (with the same number of generators) and, moreover, every minimizing curve lifts to a
minimizing curve.
Indeed,
see Theorem 4.2 in \cite{ledonne-speight} 
for the existence of free Carnot groups;
see Corollary 2.11 in \cite{ledonne-rigot} 
for the fact that the quotient is a submetry and therefore geodesics lift to geodesics.

We describe linear Casimirs on the dual of the Lie algebra. As a consequence, in the rank 3 case we characterize the symplectic foliation. Further, we apply Pontryagin maximum principle, and show that in the rank 3 case the extremal controls are either constant or periodic.

\section{Problem statement}
Let $L$ be the step 2 free-nilpotent Lie algebra with $k \geq 2$ generators:
\begin{align}
&L = L^{(1)} + L^{(2)}, \nonumber& \\
&L^{(1)} = \spann \{ X_i ~|~ i = 1, \dots, k \}, \nonumber&  \\
&L^{(2)} = \spann \{ X_{ij} ~|~ 1 \le i < j \le k \}, \nonumber & \\
&[X_i, X_j] = X_{ij}, \quad \ad X_{ij} = 0, \quad 1 \le i < j \le k, &\label{tab1} \\
&\dim L = k(k + 1)/2.& \nonumber 
\end{align}
Let $G$ be the connected simply connected Lie group with the Lie algebra $L$. We will think of $X_i, X_{ij}$ as left-invariant vector fields on $G$. 

A model of vector fields $X_i, X_{ij}$  on $$G \cong \mathbb{R}^{k(k+1)/2} = \{ (x_1, \dots, x_k;~x_{12}, \dots, x_{(k - 1)k}) \}$$ is given by
\begin{align*}
&X_i = \frac{\partial}{\partial x_i} - \sum_{j >i} \frac{x_j}{2}  \frac{\partial}{\partial x_{ij}} + 
\sum_{j <i} \frac{x_j}{2}  \frac{\partial}{\partial x_{ji}}, \quad i = 1, \dots, k,\\
&X_{ij} = \frac{\partial}{\partial x_{ij}}, \quad 1 \le i < j \le k,
\end{align*}
here we follow 
Section 2.2 in \cite{ledonne-speight}.

Let $U \subset \mathbb{R}^k$ be a compact convex set containing the origin in its interior.
We consider the following time-optimal problem:
\begin{align}
&\dot{q} = \sum^k_{i=1} u_i X_i, \quad q \in G, \quad u = (u_1, \dots, u_k) \in U, \label{p21} &\\
&q(0) = q_0 = \id, \quad q(t_1) = q_1, \label{p22} &\\
&t_1 \to \min. \label{p23}&
\end{align}
If $U = -U$, we obtain a sub-Finsler problem, and if $U$ is an ellipsoid centered at the origin, we obtain a sub-Riemannian problem. 

In the case $k = 2$, $G$ is the Heisenberg group, and solution to problem \eqref{p21}--\eqref{p23} was obtained by H. Busemann~\cite{buseman} and V. Berestovskii~\cite{ber2}. 

The sub-Riemannian case $U = \{\sum_{i=1}^k u_i^2 \leq 1\}$ was first considered by R.Brockett~\cite{brockett}, and was completely solved for $k=3$ by O.Myasnichenko~\cite{myasnich}. Some partial results for $k=4$ were obtained by L. Rizzi and U. Serres~\cite{rizzi-serres}.

We consider in greater detail the case $k = 3$, although some results concern the general case $k \ge 2$.

Existence of optimal solutions in problem \eqref{p21}--\eqref{p23} follows in a standard way  from the Rashevsky-Chow and Filippov theorems \cite{notes}.

\section{Linear Casimirs and symplectic foliation}
Before our study of extremals for the problem \eqref{p21}--\eqref{p23}, we consider Casimirs and symplectic foliation (decomposition into coadjoint orbits) on the dual $L^*$ of the Lie algebra $L$ \cite{kirillov}. This is important for our study of extremals for the problem.

Introduce linear on fibers of  the cotangent bundle $T^*G$ Hamiltonians corresponding to the basic left-invariant vector fields on $G$:
$$ h_i (\lambda) = \langle \lambda, X_i \rangle, \quad h_{ij}(\lambda) = \langle \lambda, X_{ij} \rangle, \quad \lambda \in T^*G. $$
Product rule for Lie bracket \eqref{tab1} implies the following multiplication table for Poisson bracket:
\begin{gather}
\{ h_i, h_j \} = h_{ij}, \quad \{ h_{ij}, h_l \} = \{ h_{ij}, h_{lm} \} = 0. \label{tab2}
\end{gather}
The Hamiltonians $h_i$, $h_{ij}$ can be considered as coordinates on the dual $L^*$ of the Lie algebra $L$.

Notice that the Poisson bivector (i.e., the matrix of pairwise Poisson brackets of the basis Hamiltonians $h_i$, $h_{ij}$) is determined by the skew-sym\-met\-ric matrix
\begin{gather}
M = (h_{ij}) = 
\begin{pmatrix}
0& h_{12}& \ldots &h_{1k}\\
-h_{12}& 0 & \ldots & h_{2k}\\
\vdots& \vdots& \vdots& \vdots\\
-h_{1k}& -h_{2k}& \ldots& 0
\end{pmatrix}
\in \so (k). \label{M}
\end{gather}

For a vector $a = (a_1, \dots, a_k) \in \mathbb{R}^k$, consider a linear function
$$ I_a (h) = \langle a, h \rangle = \sum^k_{i=1} a_i h_i, \quad h = (h_1, \dots, h_k) \in \mathbb{R}^k. $$
The next lemma gives conditions for a linear function $I_a$ to be a Casimir on~$L^*$.

\begin{lemma}\label{lem:Ia}
Let $M = (M_{ij}) \in \so(k)$, denote an affine subspace $$S_M = \{ h_{ij} = M_{ij} ~|~ 1 \le i < j \le k \} \subset L^*.$$ 
Then 
$$ \forall i = 1, \dots, k \quad \{ I_a, h_i \}|_{S_M} = 0 \Leftrightarrow a \in \ker M.$$
\end{lemma}
\begin{proof}
$\{ I_a, h_i \} = \{ \sum^k_{j=1} a_j h_j, h_i \} = \sum^k_{j=1} a_j \{ h_j, h_i \} = \sum^k_{j=1} a_j h_{ji} = -Ma$.
\end{proof}
By virtue of \eqref{tab2}, $h_{ij}$ are Casimirs on $L^*$. By Lemma \ref{lem:Ia}, on each $k$-dimensional subspace $\{ h_{ij} = \const \} \subset L^*$ there are $N$ linear in $h_i$ Casimirs, where $N = \dim \ker M$, and $M$ is given by \eqref{M}. This observation yields the whole symplectic foliation on $L^*$ in the case $k = 3$. Notice that in the Heisenberg case $k = 2$, the symplectic foliation consists of 2-dimensional leaves $\{ h_{12} = \const \ne 0 \}$ and 0-dimensional leaves $\{ h_{12} = 0, (h_1, h_2) = \const \}$.

\begin{theorem}\label{propos:fol}
If  the number $k$ of generators  equals $3$, then the symplectic foliation on $L^*$ consists of the following leaves:
\begin{itemize}
\item $2$-dimensional leaves
$$ h_{ij} = \const, \quad M \ne 0, \quad I_a(h) = \const,$$ 
where  $\ker M = \mathbb{R}a$, 
\item $0$-dimensional leaves
$$ h_{ij} = 0, \quad (h_1, h_2, h_3) = \const. $$
\end{itemize}
\end{theorem}
\begin{proof}
In the case $k = 3$ equality \eqref{M} 
reduces to
\begin{gather*}
M = 
\begin{pmatrix}
0& h_{12}& h_{13}\\
-h_{12}& 0 & h_{23}\\
-h_{13}& -h_{23}& 0
\end{pmatrix}
\in \so (3).
\end{gather*}
There are two possibilities:
\begin{enumerate}
\item $M \ne 0 \Leftrightarrow \dim \ker M = 1 \Leftrightarrow \rank M = 2$,
\item $M = 0 \Leftrightarrow \dim \ker M = 3 \Leftrightarrow \rank M = 0$.
\end{enumerate}

Now let us describe the symplectic foliation on $L^*$. Recall that dimension of a symplectic leaf is equal to the rank of Poisson bivector. 
The functions $h_{ij}$ are Casimirs, thus each 3-dimensional subspace $\{ h_{ij} = \const \} \subset L^*$ is foliated into symplectic leaves. If $M \ne 0$, then each symplectic leaf is 2-dimensional, and it coincides with a level surface of a linear Casimir $I_a$, $\ker M = \mathbb{R}a$. And if $M = 0$, then symplectic leaves are 0-dimensional --- points $(h_1, h_2, h_3) = \const$. 
\end{proof}

\section{Pontryagin maximum principle}
We apply Pontryagin maximum principle  (PMP) in invariant form \cite{notes} to problem \eqref{p21}--\eqref{p23}. The control-dependent Hamiltonian for this problem is $\sum^k_{i=1} u_i h_i (\lambda)$, $\lambda \in T^*G$. The Hamiltonian system of PMP
reads
\begin{align}
&\dot{h}_i = -\sum^k_{j=1} u_j h_{ij}, \quad i = 1, \dots, k, \label{Ham1}& \\
&\dot{h}_{ij} = 0, \quad 1 \le i < j \le k, \label{Ham2}&\\
&\dot{q} = \sum^k_{i=1} u_i X_i, \label{Ham3}&
\end{align}
and the maximality condition of PMP is 
\begin{gather} \label{max}
\sum^k_{i=1} u_i (t) h_i (\lambda_t) = \underset{v \in U}{\mathrm{max}} \sum^k_{i=1} v_i h_i (\lambda_t) = H(h(\lambda_t)), 
\end{gather}
where
$$ H (h_1, \dots, h_k) := \underset{v \in U}{\mathrm{max}} \sum^k_{i=1} v_i h_i$$
is the support function of the set $U$~\cite{rock}. $H$ is convex, positive homogeneous, and continuous.

Along extremal trajectories we have $H \equiv \const \ge 0.$ The abnormal case 
$$ H \equiv 0 \Leftrightarrow h_1 = \dots = h_k \equiv 0 $$
can be omitted since the distribution  $\Delta = \spann(X_1, \dots, X_k)$ satisfies the condition $\Delta^2 = \Delta + [\Delta, \Delta] = TG$, thus by Goh condition \cite{notes} all locally optimal abnormal trajectories are simultaneously normal.

So we consider the normal case: $H \equiv \const > 0$. In view of homogeneity of the vertical part
\eqref{Ham1}, \eqref{Ham2} of the Hamiltonian system of PMP, we will assume that $H \equiv 1$ along extremal trajectories.

From now on we suppose additionally that the set $U$ is strictly convex. Then the maximized Hamiltonian $H$ is $C^1$-smooth on $\mathbb R^k \setminus \{0\}$, and maximum in \eqref{max} is attained at the control $u = \nabla H = (\partial H/\partial h_1, \dots, \partial H / \partial h_k)$ \cite{rock}. Denote $H_i = \partial H / \partial h_i$, $i = 1, \dots, k$. Then the vertical subsystem of the Hamiltonian system reads as follows:
\begin{gather} \label{Hamax}
\dot{h}_i = -\sum^k_{j=1} h_{ij} H_j, \quad \dot{h}_{ij} = 0, \quad 1 \le i < j \le k. 
\end{gather}
In addition to obvious integrals $h_{ij}$, the system \eqref{Hamax} has also the integral $H$ and the linear integrals
$$ I_a (h), \quad a \in \ker M = \ker (h_{ij}). $$
The last claim follows from Lemma \ref{lem:Ia}.

\section{Extremals in the case $k = 3$}
Let $k = 3$. Then the skew-symmetric matrix $M = (h_{ij})$ has a nonzero kernel, this allows us to characterize solutions to system \eqref{Hamax} as follows.

If $M = 0$, then all solutions to system \eqref{Hamax} are constant.

If $M \ne 0$, then $\dim \ker M = 1$, and we have the following description.
\begin{theorem} \label{th:period}
Let $L$ be a step-2 Carnot algebra with 3 generators.
Let $0 \ne M = (h_{ij})  \in \so(3)$. Suppose that $U$ is strictly convex and compact, and contains the origin in its interior. Let $\ker M = \mathbb{R}a$, $0 \ne a \in \mathbb{R}^3$. Then for any  $h^0 \in H^{-1}(1)$ the solution $h(t)$ to system \eqref{Hamax} with the initial condition $h(0) = h^0$ is unique and $C^2$-smooth. Moreover:
\begin{enumerate}
\item if $\nabla H (h^0) \parallel a$, then $h(t) \equiv h^0$,
\item if $\nabla H (h^0) \nparallel a$, then $h(t)$ is a regular periodic planar curve.
\end{enumerate}
\end{theorem}
\begin{proof}
As we noticed in Sec.~\ref{sec:intro}, we can assume that $L$ is the step-2 free-nilpotent Lie algebra with 3 generators.
Consider the constrained optimization problem 
\begin{gather}\label{p2_cond}
I_a (h) \to \extr, \quad H(h) \leq 1.
\end{gather}
Since the polar set  $U^{\circ} = \{ H \leq 1 \}$ is compact and $0 \in \interior U^{\circ}$, this problem has solutions 
$$ I^{\max}_a = \max I_a |_{U^{\circ}} >  I^{\min}_a = \min I_a |_{U^{\circ}}. $$

The condition $\nabla H (h) \parallel a$ is a 
necessary and sufficient condition for global extremum in the convex optimization problem \eqref{p2_cond}. 

1.  Let $\nabla H (h^0) \parallel a$. 
Then $I_a(h^0) \equiv I_a(h(t)) \equiv I_a^{\max}$ or $I_a^{\min}$.
Thus $\nabla H (h(t)) \parallel a$ and $\dot h(t)  = M \nabla H(h(t))\equiv 0$, whence $h(t) \equiv h^0$. 

2. Let $\nabla H (h^0) \nparallel a$. Then the point $h^0$ is not a solution to problem~\eqref{p2_cond}, thus $I_a (h^0) \in (I^{\min}_a, I^{\max}_a)$.
The curve
$$
\Gamma = \{ H(h) = 1 \} \cap \{ I_a(h) = I_a(h^0) \} \subset \mathbb R_{h_1, h_2, h_3}^3
$$
is compact and planar. Any $h \in \Gamma$ satisfies the inclusion $I_a(h) \in (I_a^{\min}, I_a^{\max})$,  thus it is not a solution to problem~\eqref{p2_cond}, so $\nabla H(h) \nparallel a$. Consequently, $\Gamma$ is a $C^1$-regular curve diffeomorphic to $S^1$. 

Choose coordinates in $\mathbb R_{h_1, h_2, h_3}^3$ such that $a = (0, 0, 1)$, then $I_a(h) = h_3$. Parametrize the curve $\Gamma$ as follows: 
$$
h_1 = f_1(\f), \quad h_2 = f_2(\f), \quad h_3 \equiv \const, \qquad \f \in S^1,
$$
where $f_1, f_2 \in C^1(S^1)$. In this parametrization ODE~\eqref{Hamax} reads as
\be{dotf}
\dot \f = f(\f), \qquad \f \in S^1,
\ee
where $f \in C^1(S^1)$ and $f(\f) \neq 0$ for all $\f \in S^1$. ODE~\eqref{dotf} can uniquely be solved for any initial data by separation of variables, thus it has a unique solution $\f(t) \in C^2(S^1)$ for any Cauchy problem $\f(0) = \f^0$. Thus ODE~\eqref{Hamax} has also a unique solution $h(t) \in C^2$ for the Cauchy problem $h(0) = h^0$.

Further, $h(t) \in \Gamma$ and $\dot h(t) = M \nabla H(h(t)) \neq 0$ for all $t$, thus there exists $T > 0$ such that $h(T) = h^0$. By uniqueness of $h(t)$, it is $T$-periodic.
\end{proof}

\begin{corollary}\label{cor:periodic}
Let $L$ be a step-2 Carnot algebra with 3 generators.
Suppose that $U$ is strictly convex and compact, and contains the origin in its interior.
Then all  normal extremal controls are constant or periodic, and continuous.
\end{corollary}
\begin{proof}
Use $u(t) = \nabla H (h(t))$, and apply  Th.~\ref{th:period}. 
\end{proof}

It is known that in the sub-Riemannian case $U = \{ \sum^3_{i=1} u^2_i \le 1 \}$ the extremal controls are given by constant or periodic trigonometric functions \cite{myasnich}.

In a recent paper \cite{hakavuori}, E. Hakavuori proved that for step 2 sub-Finsler Carnot groups with strictly convex norms, only lines are infinite geodesics. Thus in the case $k = 3$, in view  of Th. \ref{th:period}, we have corollaries:
\begin{enumerate}
\item
constant controls are optimal (trivial),
\item
for a periodic nonconstant control $u(t)$ there exists $t_1>0$ such that $u|_{[0, t_1]}$ is not optimal (nontrivial).
\end{enumerate}

\section{General free-nilpotent Lie groups}
It is interesting to look for a generalization of Cor.~\ref{cor:periodic} in the perspective of arbitrary free-nilpotent Lie groups with $k\geq 2$ generators and $s\geq 1$ steps.

In the Abelian case $s = 1$, $k\geq 2$ we have the problem
\begin{align*}
&\dot x = u, \qquad u \in U \subset \mathbb R^k, \quad x \in \mathbb R^k, \\
&x(0) = 0, \qquad x(t_1) = x_1, \\
&t_1 \to \min,
\end{align*}
which obviously has only constant extremal controls.

In the cases $s = 2$, $k = 2$ and $k = 3$ problem~\eqref{p21}--\eqref{p23} has either constant or periodic extremal controls, see~\cite{ber2} and Cor.~\ref{cor:periodic}. Although, in the next case $s=2$, $k=4$ there are possible nonperiodic extremal controls, see the following example.

\begin{example}
Let $s = 2$, $k = 4$, $U = \{\sum_{i=1}^4 u_i^2 \leq 1\}$. Then $H = \sum_{i=1}^4 h_i^2$, and Hamiltonian system~\eqref{Hamax} reads
$$
\dot h_i = - \sum_{i=1}^4 h_{ij} h_j, \qquad \dot h_{ij}=0.
$$
If
$$
M = (h_{ij}) = 
\left(\begin{array}{cccc}
0 & \alpha & 0 & 0 \\ 
-\alpha & 0 & 0 & 0\\
0 & 0& 0 & \beta \\
0 & 0& -\beta & 0 
\end{array}\right), 
\qquad \alpha/\beta \in \mathbb R \setminus \mathbb Q,
$$
then $u = (h_1, \dots, h_4)$ is not periodic for general initial conditions.
\end{example}

In the case $s= 3$, $k = 2$, $U = \{u_1^2+u_2^2 \leq 1\}$, optimal controls are given by Jacobi's elliptic functions~\cite{dido_exp}, and they are of the following classes:
\begin{itemize}
\item
constant,
\item
periodic,
\item asymptotically constant (with constant limits as $t \to \pm \infty$).
\end{itemize}

It would be interesting to characterize similarly optimal controls in the cases $s=3$, $k \geq 3$ and $s\geq 4$. In these cases, if $U = \{\sum_{i=1}^k u_i^2 \leq 1\}$, the normal Hamiltonian system of Pontryagin maximum principle is not Liouville integrable~\cite{borisov, 2358}. 




\section*{Acknowledgment}

Yu. S. thanks Enrico Le Donne and Lev Lokutsievskiy for fruitful discussions of sub-Finsler geometry.
He is also grateful to Department of Mathematics and Statistics, University of Jyv\"askyl\"a, Finland, for hospitality and excellent conditions for work.

The author thanks anonymous reviewers whose comments improved exposition and allowed to add some references.


\end{document}